\documentclass{article}
\usepackage{amsthm,amsmath,amsfonts,amssymb}

\newlength{\hchng}
\newlength{\vchng}
\newtheorem{thm}{Theorem}[section]

\newtheorem{lemma}[thm]{Lemma}

\newtheorem{preremark}[thm]{Remark}
\newenvironment{remark}{\begin{preremark}\rm}{\medskip \end{preremark}}
\numberwithin{equation}{section}

\newcommand{\norm}[1]{\left\Vert#1\right\Vert}

\newcommand{\R}{\mathbb R}

\newcommand{\grad} {\nabla}
\newcommand{\lap} {\triangle}

\newcommand{\dd} {\; \mathrm{d}}

\DeclareMathOperator*{\osc}{osc}

\DeclareMathOperator{\tr}{tr}

\title{Upper bounds for multiphase composites in any dimension}
\author{Luis Silvestre}

\begin{document}
\maketitle

\begin{abstract}
We prove a rigorous upper bound for the effective conductivity of an isotropic composite made of several isotropic components in any dimension. This upper bound coincides with the Hashin Shtrikman bound when the volume ratio of all phases but any two vanish.
\end{abstract}

\section{Introduction}
The effective conductivity of a composite material made of two isotropic phases with given conductivities and volume ratios is bounded below and above by the Hashin-Shtrikman bounds \cite{hashin1962variational}. These bounds are exact. For any value of the parameters there is some micro structure that realizes the bounds. The Hashin-Shtrikman bounds can be extended to composites of more than two phases, but they are not exact for some values of the parameters. As it was pointed out by Milton \cite{milton1981concerning}, the Hashin-Shtrikman upper bound for three phase composites for which the most composite phase has volume ratio zero, does not coincide with the bound for the two other phases only. This is highly counterintuitive, and strongly suggest that there should be better bounds for some values of the parameters.

In two dimensions, the Hashin-Shtrikman bounds for multiphase composites were refined by Nesi \cite{MR1363001} and Astala and Nesi \cite{MR2020365}. Exact bounds for the two dimensional case were proposed by Cherkaev in \cite{cherkaev2009bounds}. In more than two dimensions, there is no bound that refines the Hashin-Shtrikman ones. The purpose of this paper is to progress in that direction.

In this paper we prove a new upper bound for the effective conductivity of an isotropic composite with an arbitrary number of isotropic phases in arbitrary dimension that refines the Hashin-Shtrikman bound. The bound is not exact except when Hashin-Shtrikman bounds are exact. In fact the bound seems to be very rough, but it does satisfy the right asymptotics. If we let the volume ratio of all phases but two go to zero, the bound converges to the corresponding two-phase Hashin-Shtrikman, for any choice of the two phases to keep. This provides the first rigorous proof, in dimension larger than two, of the intuitively obvious fact that a phase with a negligible volume ratio has a negligible effect in the possible effective conductivities.

We consider periodic composite materials made from isotropic phases. The microscopic structure is given by a periodic function $\sigma$. Typically, for a composite with $K$ phases, the function $\sigma$ has the following structure.
\begin{equation} \label{e:a-coefficient}
\sigma(x) = \begin{cases}
\sigma_1 & \text{in } A_1 \\
\sigma_2 & \text{in } A_2 \\
\cdots \\
\sigma_K & \text{in } A_K 
\end{cases}
\end{equation}

The sets $A_1, \cdots, A_K$ are a partition of the unit cube $Q$ of $\R^n$ and $\sigma(x)$ is extended periodically outside $Q$. We denote by $\mu_i$ the measure of each set $A_i$. Obviously, since $\bigcup_i A_i = Q$, $\mu_1+\dots+\mu_K=1$.

The effective conductivity associated with $\sigma$ is computed via the cell problem:
\begin{equation} \label{e:cell-problem}
\langle A v , v \rangle = \min_{u \in H^1_p(Q)} \int_Q \sigma(x) |v+\grad u|^2 \dd x
\end{equation}
where $H^1_p(Q)$ stands for the $H^1$ functions in the unit cube $Q$ with periodic boundary conditions.

In this article we will concentrate in finding bounds for the trace of $A$ that depend on the values of $\mu_i$ and $\sigma_i$ but not on the particular structure $A_i$. In case the composite has cubic symmetry, then the effective conductivity is isotropic. In that case, our bounds are just bounds of the effective conductivity.

We define
\[ \bar \sigma = \frac 1 n \tr A = \frac 1 n \sum_{i=1,\dots, n} \langle A e_i, e_i \rangle. \]
Our purpose is to find upper and lower bounds for $\bar \sigma$.

Clearly, the multiphase composite problem is a particular case of the problem of finding bounds for $\bar \sigma$ when all that is known of the periodic coefficient $\sigma$ is its distribution function
\[ F(t) = |\{ x : \sigma(x) > t\}|.\]
Thus, in this work we will consider the problem of finding an upper bound of $\bar \sigma$ that depends on $F$ only. It is important to remember that $\inf \sigma$, $\sup \sigma$ and an integral of the form
\[ \int_Q G(\sigma(x)) \dd x, \]
for any function $G$, depend on the distribution of $\sigma$ only. 

The following theorem provides an upper bound that refines the Hashin-Shtrikman bounds.

\begin{thm} \label{t:upperbound}
For any value of the parameter $S>0$, the following upper bound holds:
\[ \bar \sigma \leq H+E \]
where 
\begin{align*}
H &= - (n-1)S + \left( \int_Q (\sigma(x)+(n-1)S)^{-1} \right)^{-1} \\
E &= \frac {C (\osc \sigma)^2 \left( \int_Q (\sigma(x)+(n-1)S)^{-1} \right)^{-2} }{(\inf \sigma + (n-1) S)^2(\sup \sigma + (n-1) S)^2}  \int_S^{\infty} F(t) (1- \log F(t))^2 \dd t
\end{align*}
where $C$ is a constant depending only on dimension.
\end{thm}

We will not provide an explicit expression for the constant $C$ in Theorem \ref{t:upperbound}. It is a constant that depends on dimension only. It is related to the BMO estimates for the Laplace equation in the unit cube with a bounded right hand side. The best constant for these estimates is difficult to compute.

The proof of Theorem \ref{t:upperbound} uses the idea (originally from \cite{silvestre2007characterization}) that an upper bound can be obtained by reducing the set of possible vector fields in the cell problem to those that are the gradient of a scalar potential function. The term $E$ in the upper bound of Theorem \ref{t:upperbound} is obtained applying harmonic analysis estimates to this potential function. With this approach, upper bounds are more natural to obtain than lower bounds (unlike other approaches for bounds like the translation method). In this article we leave the question of a similar lower bound open. 

Note that for any $S \leq \sup \sigma$, 
\[ (n-1) S \leq \left( \int_Q (\sigma(x)+(n-1)S)^{-1} \right)^{-1} \leq (n-1) S + \sup \sigma. \]
Therefore we can simplify the expression of $E$ for the bound in theorem \ref{t:upperbound} to
\begin{equation} \label{e:Esimplified}
 E \leq \frac {C (\osc \sigma)^2}{(\inf \sigma + (n-1) S)^2}  \int_S^{\infty} F(t) (1- \log F(t))^2 \dd t
\end{equation}

Let us see what bounds Theorem \ref{t:upperbound} provides in the case of a three phase composite. In this case the function $\sigma$ has the form
\[ \sigma(x) = \begin{cases}
\sigma_1 & \text{in } A_1 \\ 
\sigma_2 & \text{in } A_2 \\ 
\sigma_3 & \text{in } A_3
\end{cases} \]
where $|A_1|=\mu_1$, $|A_2|=\mu_2$ and $|A_3|=\mu_3$. 

If we choose $S = \sigma_3$, then the term $E$ in Theorem \ref{t:upperbound} equals to zero, and we recover the classical Hashin-Shtrikman bound for three phase composites.

If we choose $S =\sigma_2$, then we obtain the following bound using the upper bound \eqref{e:Esimplified},
\begin{equation} \label{e:refined-three-phase}
 \bar \sigma \leq -(n-1)\sigma_2 + \left( \sum_{i=1,2,3} \mu_i (\sigma_i + (n-1)\sigma_2)^{-1} \right)^{-1} + \frac {C (\sigma_3-\sigma_1)^2(\sigma_3-\sigma_2) \mu_3 (1- \log \mu_3)^2 }{(\sigma_1 + (n-1) \sigma_2)^2}.
\end{equation}

This bound refines the Hashin-Shtrikman bound when the volume ratio $\mu_3$ is small. Indeed, as $\mu_3 \to 0$, \eqref{e:refined-three-phase} converges to the two phase Hashin-Shtrikman bound corresponding to the first and second phase only.

The organization of the paper is as follows. In section \ref{s:trivial} we prove the trivial upper bound that $\bar \sigma$ is bounded above by the average of $\sigma$. In section \ref{s:hashin-shtrickman} we provide a proof of the classical Hashin-Shtrikman bounds, and in section \ref{s:new-bounds} we prove the new bounds of Theorem \ref{t:upperbound}. Even though the bounds obtained in sections \ref{s:trivial} and \ref{s:hashin-shtrickman} are well known, we use these two sections to introduce the ideas that will be used to obtain the new bounds in section \ref{s:new-bounds}. These ideas originated in \cite{silvestre2007characterization}. The inequalities obtained in section \ref{s:hashin-shtrickman} are the basis of the proof of Theorem \ref{t:upperbound} in section \ref{s:new-bounds}.

\section{The trivial upper bound}
\label{s:trivial}
Let us analyze the formula for $\bar \sigma$.
\begin{align*}
\bar \sigma &= \frac 1 n \sum_{i=1,\dots, n} \langle A e_i, e_i \rangle \\
&= \frac 1 n\sum_{i=1,\dots, n} \min_{u_i \in H^1_p(Q)} \int_Q \sigma(x) |e_i + \grad u_i|^2 \dd x \\
&= \frac 1 n\min_{U \in H^1_p(Q,\R^n)} \int_Q \sigma(x) |I + DU|^2 \dd x
\end{align*}
were $H^1_p(Q,\R^n)$ stands for the space of periodic function is $Q$ with values in $\R^n$ and finite norm in $H^1$. Indeed $U = (u_1,\dots,u_n)$. Therefore, the quantity that we have to estimate is
\begin{equation} \label{e:vector-integral}
\bar \sigma = \frac 1n \min_{U \in H^1_p(Q,\R^n)} \int_Q \sigma(x) |I + DU|^2 \dd x
\end{equation}

One way to obtain an upper bound for \eqref{e:vector-integral} is by using a sample vector field $U$. The trivial upper bound is what we obtain by setting $U \equiv 0$.
\[ \bar \sigma \leq \int_{Q} \sigma(x) \dd x \]

In the case of a composite with $K$ phases, this trivial upper bound becomes $\sum_{i=1,\dots,K} \mu_i \sigma_i$.

This trivial upper bound is not achievable as soon as there is more than one different phase.

\section{The Hashin-Shtrikman upper bound}
\label{s:hashin-shtrickman}

We will start by showing a simple way to obtain the standard Hashin-Shtrikman bound for $\bar \sigma$. The idea of this proof was first given in \cite{silvestre2007characterization}. This idea was extended in \cite{liu2007hashin} where it was used to obtain the full Hashin-Shtrikman bounds for multiphase composites and study their attainability conditions.

The main idea is to obtain an upper bound for \eqref{e:vector-integral} by reducing the set where we look for a minimum to only those vector fields $U$ that are the gradient of a potential $p$. Therefore
\begin{equation} \label{e:scalar-integral}
 \bar \sigma \leq \frac 1 n \min_{p \in H^2_p(Q)} \int_Q \sigma(x) |I + D^2 p|^2 \dd x.
\end{equation}

The following elementary relations will be used. Note that the first identity plays a key role in the usual proof of the Hashin-Shtrikman bounds by the method of compensated compactness (\cite{luriecherkaev1984}, \cite{luriecherkaev1986} and \cite{tartar1985}. See also \cite{kohnmilton1986}).
\begin{itemize}
\item For any periodic function $p$ in $H^2_p$ we have
\begin{equation} \label{e:identity-hessian-laplacian}
\int_Q |D^2p|^2 \dd x = \int_Q |\lap p|^2 \dd x
\end{equation}

\item For every matrix $M$, the following inequality holds
\begin{equation} \label{e:matrix-inequality}
|M|^2 \geq \frac 1 n |\tr M|^2.
\end{equation}
In particular, for $M = D^2 p$, $|D^2p|^2 \geq \frac 1 n |\lap p|^2$. The equality holds only if $D^2 p$ is a scalar matrix.

\item For any periodic function $p$
\begin{equation} \label{e:lap-integral-zero}
\int_Q \lap p = 0
\end{equation}
Moreover, any function with zero average in the cube is the Laplacian of some potential $p$.
\end{itemize}

Let $S$ be an arbitrary constant. We apply these relations to \eqref{e:scalar-integral} and obtain
\begin{equation} \label{e:I1plusI2}
\begin{aligned}
\bar \sigma &=  \frac 1 n\min_{p \in H^2_p(Q)} \int_Q \sigma(x) |I + D^2 p|^2 \dd x \\
&= \min_{p \in H^2_p(Q)} I_1(\lap p) + I_2(D^2 p)
\end{aligned}
\end{equation}
where
\begin{align}
I_1(\lap p) &= \frac 1 n \int_Q \sigma(x) n + 2 \sigma(x) \lap p + S |\lap p|^2 + \frac 1 n (\sigma(x)-S) |\lap p|^2 \dd x, \label{e:I1} \\
I_2(D^2 p) &= \frac 1 n \int_Q (\sigma(x)-S) \left( |D^2 p|^2 - \frac 1n (\lap p)^2 \right) \dd x. \label{e:I2}
\end{align}

The advantage of these estimate is that \eqref{e:I1} depends only on $\lap p$. The value of $\lap p(x)$ can be chosen arbitrarily for every point $x$ with the only condition that it has average zero on $Q$.

In order to find an upper bound for $\bar \sigma$, we will choose the $p$ that minimizes $I_1(\lap p)$, and then we will estimate the remainder $I_2(D^2 p)$ using elliptic estimates for the obtained function $p$.

Since $\lap p$ has average zero on $Q$, we can rewrite $I_1$ as
\begin{align*}
I_1(\lap p) &=  \frac 1 n \int_Q \sigma(x) n + 2 (\sigma(x)+(n-1)S) \lap p + S |\lap p|^2 + \frac 1 n (\sigma(x)-S) |\lap p|^2 \dd x,\\
&= \int_Q \sigma(x) + (\sigma(x)+(n-1)S) \left( \left(1 + \frac {\lap p} n \right)^2 - 1 \right) \dd x, \\
&= - (n-1)S + \int_Q (\sigma(x)+(n-1)S)\left(1 + \frac {\lap p} n \right)^2 \dd x
\end{align*}

Now we will find the value of $\lap p$ that makes the value of $I_1$ smallest. We compute the first variation of $I_1$:
\[ DI_1 \cdot h = \int_Q (\sigma(x)+(n-1)S) \left(1 + \frac {\lap p} n \right) h(x) \dd x, \]
which has to be zero for any function $h$ with average zero on $Q$. This means that there is a constant $L$ such that
\[ (\sigma(x)+(n-1)S) \left(1 + \frac {\lap p} n \right) = L. \]
We can compute the exact value of $L$ using again that $\int_Q \lap p = 0$.
\begin{equation} \label{e:L}
 L = \left( \int_Q (\sigma(x)+(n-1)S)^{-1} \right)^{-1}.
\end{equation}
Now, for this optimal choice of $\lap p$, we have
\begin{align}
I_1(\lap p) &=  - (n-1)S + \int_Q L \left(1 + \frac {\lap p} n \right) \dd x \\
&= - (n-1)S + L = - (n-1)S + \left( \int_Q (\sigma(x)+(n-1)S)^{-1} \right)^{-1} \label{e:bound-of-I1}
\end{align}

Let us now concentrate on the value of $I_2(D^2 p)$ for this same choice of function $p$ that minimizes $I_1(\lap p)$.

We first observe that from \eqref{e:matrix-inequality}, $\left( |D^2 p|^2 - \frac 1n (\lap p)^2 \right) \geq 0$. Therefore, we can disregard the points $x$ where $\sigma(x) \leq S$:
\begin{equation} \label{e:I2-positivepart}
 I_2(D^2 p) \leq \frac 1 n \int_Q (\sigma(x)-S)^+ \left( |D^2 p|^2 - \frac 1n (\lap p)^2 \right) \dd x.
\end{equation}

In particular if we choose $S=\sup \sigma$ ($=\sigma_K$ in the case of a multi-phase composite) we make $I_2(D^2 p)=0$, we obtain that $\bar \sigma \leq I_1(\lap p)$ and we recover the Hashin-Shtrikman bound
\begin{equation} \label{e:hs}
 \bar \sigma \leq - (n-1) \sup \sigma + \left( \int_Q (\sigma(x)+(n-1) \sup \sigma)^{-1} \right)^{-1}
\end{equation}

The attainability condition is that the minimal vector field $U$ in \eqref{e:vector-integral} is the gradient of a scalar potential $p$ and also $D^2 p$ is a constant scalar matrix in all phases except the most conductive one. The range of parameters for which one can guarantee the attainability of the bounds \eqref{e:hs} is analyzed in \cite{liu2007hashin}. 

\section{New bounds}
\label{s:new-bounds}
In the optimal configuration for the Hashin-Shtrikman bounds, the most conductive phase plays the role of the matrix and all other phases are inclusions. The Hashin-Shtrikman bound cannot be optimal for multiphase composites in all range of parameters. Indeed, if we apply \eqref{e:hs} to a three phase composite and make the volume ratio of the most conductive phase $\mu_3$ approach zero, the bound \eqref{e:hs} still depends on $\sigma_3$ in the limit, even though the third phase is not present. When the most conductive phase has a very small volume ratio, it is more convenient to choose a different value of the arbitrary parameter $S$ and estimate \eqref{e:I2}.

In this section we prove theorem \ref{t:upperbound}. The proof refines the estimate from section \ref{s:hashin-shtrickman}. In particular it is based on the estimates \eqref{e:I1plusI2}, \eqref{e:I1} and \eqref{e:I2-positivepart}. In order to find an upper bound for $I_2(D^2 p)$ we will need the following technical Lemma about BMO functions.

\begin{lemma} \label{l:bmo}
Let $f: Q \to \R$ be a BMO function in $Q$ with average zero. Then for any subset $A \subset Q$,
\begin{equation} \label{e:BMO}
\int_{A} |f|^2 \dd x \leq C \norm{f}_{BMO}^2 (1-\log |A|)^2 |A|
\end{equation}
where $|A|$ stands for the measure of $A$ and $C$ depends only on dimension.
\end{lemma}

Note that \eqref{e:BMO} also holds if $f$ is vector or matrix valued, since the inequality can be applied component-wise. We defer the proof of this Lemma to the appendix. We now prove the main theorem.

\begin{proof}[Proof of Theorem \ref{t:upperbound}]
As we mentioned above, we will find an upper bound for $\bar \sigma$ using \eqref{e:I1plusI2}. Our upper bound corresponds to one test function $p$, which may or may not be optimal for \eqref{e:I1plusI2}. In fact, we will choose the function $p$ which is optimal only for the value of $I_1(\lap p)$ and then we will estimate the value of $I_2(D^2 p)$ from above.

Note that 
\[ \left( |D^2 p|^2 - \frac 1 n |\lap p|^2 \right) = |D^2 p - \frac 1n (\tr D^2p) I|^2. \]
Thus, $I_2(D^2 p)$ is the integral of the squared norm of the traceless part of $D^2 p$.

Recall from section \ref{s:hashin-shtrickman}, that the potential $p$ that minimizes $I_1(\lap p)$ solves 
\begin{equation} \label{e:laplace-equation}
\lap p = n L (\sigma(x) + (n-1)S)^{-1} - n =: \theta(x)
\end{equation}

For this $p$, it is simple to estimate the $L^2$ norm of $D^2 p - \frac 1n \lap p I$ using the Fourier transform. Indeed
\[
\left(D^2 p - \frac 1n \lap p I\right)^\wedge (\xi) =  \left( -\frac{\xi \otimes \xi}{|\xi|^2} + \frac 1n I \right) |\xi|^2 \hat p(\xi)
\]
therefore 
\[ \norm{D^2 p - \frac 1n \lap p I}_{L^2} \leq \left(\frac{n-1}{n}\right)^{1/2} \norm{\theta}_{L^2} \]

An $L^2$ estimate for the integrand only provides a rough estimate for \eqref{e:I2-positivepart} that would only take into account the maximum value of $(\sigma(x)-S)$. On the other hand, $\theta$ is a bounded function. We can use Calderon-Zygmund theory to obtain that $D^2 p$ is a BMO function whose norm depends on the oscillation of $\theta$ only (see \cite{stein1993}).
\begin{equation} \label{e:BMO-estimate}
\norm{D^2 p - \frac 1n \lap p I}_{BMO(Q)} \leq C \osc_Q \theta
\end{equation}
where $C$ is a constant depending only on dimension.

We compute the oscillation of $\theta$ from \eqref{e:laplace-equation}:
\begin{equation} \label{e:osc-theta}
\begin{split}
\osc_Q \theta &= nL \left( \frac 1 {\inf \sigma + (n-1) S} - \frac 1 {\sup \sigma + (n-1)S} \right)  \\ 
 &= \frac {n }{(\inf \sigma + (n-1) S)(\sup \sigma + (n-1) S)} L \osc \sigma
\end{split}
\end{equation}
Recall that $L$ was given in \eqref{e:L}.

Combining  \eqref{e:osc-theta} with \eqref{e:BMO-estimate} and \eqref{e:BMO} we can obtain an upper bound for $I_2(D^2u)$ that depends on the distribution function $F$ of $\sigma$ only. From Fubini's theorem:
\begin{align*}
I_2(D^2 p) &\leq \frac 1 n \int_Q (\sigma(x)-S)^+ \left( |D^2 p|^2 - \frac 1n (\lap p)^2 \right) \dd x, \\
&\leq \frac 1 n \int_S^{+\infty} \int_{\sigma(x)>t} \left( |D^2 p|^2 - \frac 1n (\lap p)^2 \right) \dd x \dd t  \\
&\leq \frac 1 n \norm{D^2 p - \frac 1n \lap p I}_{BMO}^2 \int_S^{+\infty} F(t) (1- \log F(t))^2 \dd t  \\
&\leq \frac {C L^2 (\osc \sigma)^2 }{(\inf \sigma + (n-1) S)^2(\sup \sigma + (n-1) S)^2} \int_S^{+\infty} F(t) (1- \log F(t))^2 \dd t  \\
\end{align*}

Adding this bound for $I_2(D^2p)$ with the bound \eqref{e:bound-of-I1} for $I_1(\lap p)$ we finish the proof of Theorem \ref{t:upperbound}.
\end{proof}

\begin{remark}
The constant $C$ in \eqref{e:BMO-estimate} must be smaller for the traceless part of the Hessian: \\ $\norm{D^2 p - \frac 1n \lap p I}_{BMO(Q)}$ than for the whole Hessian $\norm{D^2 p}_{BMO(Q)}$. It is hard to know the best constants though.
\end{remark}

\begin{remark}
The obtained bounds for $I_2(D^2 p)$ seem to be very crude. In \eqref{e:BMO-estimate} we estimate the BMO norm of $D^2 p$ with respect to the $L^\infty$ bound of $\lap p$. An estimate in $L^\infty$ for $D^2 p$ is known to be false for the Laplace equation. On the other hand, as it is pointed out in \cite{cherkaev2009bounds}, if there exists a composite with piecewise differentiable interfaces where the bounds for \eqref{e:vector-integral} are realized, $|DU|$ should be bounded.
\end{remark}

\section{Appendix: The proof of Lemma \ref{l:bmo}}

In order to prove Lemma \ref{l:bmo}, we will use the following classical result by F. John and L. Niremberg about BMO functions (for a proof see \cite{MR0131498}).

\begin{lemma} \label{l:john-niremberg}
Let $f$ be a BMO function on $Q = [0,1]^n$ with average zero. Then the following inequality holds for the distribution of $f$:
\[ | \{|f|>\sigma\}| \leq B e^{\frac {-b \sigma}{\norm{f}_{BMO}} } \]
where $b$ and $B$ are constants depending only on the dimension $n$.
\end{lemma}

\begin{proof}[Proof of Lemma \ref{l:bmo}]
Let $T>0$ be a parameter which will be chosen later.

We use Lemma \ref{l:john-niremberg} to estimate the integral.
\begin{align*}
\int_{A} |f|^2 \dd x &= \int_0^\infty 2t | \{ |f|> t \} \cap A| \dd t \\
&\leq \int_0^\infty 2t \min\left( B e^{\frac {-b t}{\norm{f}_{BMO}} }, |A| \right) \dd t
\end{align*}
For any $T \in (0,+\infty)$, 
\begin{align*}
\int_{A} |f|^2 \dd x &\leq \int_0^T 2t |A| \dd t + \int_T^\infty 2t B e^{\frac {-b t}{\norm{f}_{BMO}} } \dd t \\
&= T^2 |A| + 2B\left(T + \frac{\norm{f}_{BMO}} b \right) \frac{\norm{f}_{BMO}} b e^{\frac {-b T}{\norm{f}_{BMO}} } \\
\intertext{Choosing now $T = - \frac{||f||_{BMO}}{b} \log |A|$,}
&\leq C ||f||_{BMO}^2 (1-\log |A|)^2 |A|
\end{align*}
for a constant $C$ depending only on dimension.
\end{proof}

\bibliographystyle{plain}
\bibliography{tpc}
\end{document}